\documentclass[leqno,11pt]{amsart}

\usepackage{verbatim}
\usepackage{amsmath,amscd,amsthm,amsxtra,amssymb}
\usepackage{epsfig,graphics,color,colortbl}
\usepackage{amssymb,latexsym}
\usepackage{mathrsfs}
\usepackage{soul}   
\usepackage[poly,all]{xy}
\usepackage{marginnote}
\usepackage[colorlinks=true, pdfstartview=FitV, linkcolor=blue,citecolor=blue,urlcolor=blue]{hyperref}

\usepackage[usenames,dvipsnames,svgnames,table]{xcolor}

\newtheorem{thm}{Theorem}[section]
\newtheorem{prop}[thm]{Proposition}

\newtheorem{lem}[thm]{Lemma}

\theoremstyle{definition}
\newtheorem{defn}[thm]{Definition}
\newtheorem{Ex}[thm]{Example}

\theoremstyle{remark}
\newtheorem{Rmk}[thm]{Remark}

\newenvironment{red}
{\relax\color{red}}
{\hspace*{.5ex}\relax}

\newcommand{\ber}{\begin{red}}
\newcommand{\er}{\end{red}}

\newcommand{\berE}{\begin{red}{}\marginnote{\fbox{\scshape\lowercase{E}}}{}}

\numberwithin{equation}{section}

\newcommand{\Z}{\mathbb{Z}}
\newcommand{\Q}{\mathbb{Q}}

\newcommand{\g}{\mathfrak{g}}
\newcommand{\h}{\mathfrak{h}}

\newcommand{\Hom}{\mathrm{Hom}}

\newcommand{\Ht}{{\rm ht}}

\newcommand{\rlQ}{\mathsf{Q}}   
\newcommand{\wlP}{\mathsf{P}}   
\newcommand{\weyl}{\mathsf{W}}  
\newcommand{\cmA}{\mathsf{A}}  
\newcommand{\tf}{\tilde{f}}  
\newcommand{\te}{\tilde{e}}  

\newcommand{\wt}{\mathrm{wt}} 		
\newcommand{\ep}{\varepsilon}  		
\newcommand{\ph}{\varphi}  		

\newcommand{\sg}{\mathfrak{S}}   

\newcommand{\lan}{\langle} 	
\newcommand{\ran}{\rangle}	

\newcommand{\SST}{\mathsf{SST}}					

\newcommand{\sh}{\mathsf{sh}}

\newcommand{\La}{\Lambda}
\newcommand{\la}{\lambda}
\newcommand{\al}{\alpha}

\newcommand{\cs}{\mathsf{s}} 			
\newcommand{\cw}{\mathsf{w}} 			
\newcommand{\ch}{\mathrm{ch}}					
\newcommand{\qdim}{\dim_q}					
\newcommand{\ev}{\mathsf{ev}}					
\newcommand{\cc}{\mathsf{c}}					
\newcommand{\CC}{\mathsf{C}}					
\newcommand{\PP}{\mathsf{P}}					
\newcommand{\cont}{\mathrm{cont}}					
\newcommand{\pr}{\mathsf{pr}}					
\newcommand{\id}{\mathrm{id}}					
\newcommand{\OO}{\mathcal{O}}					
\newcommand{\oo}{\mathfrak{e}}					
\newcommand{\lcm}{\mathrm{lcm}}					
\newcommand{\nz}{\mathsf{m}}					
\newcommand{\po}{\mathsf{d}}					

\begin{document}

\title[Crystals, semistandard tableaux and cyclic sieving phenomenon]
{Crystals, semistandard tableaux and cyclic sieving phenomenon}

\author[Young-Tak Oh]{Young-Tak Oh}
\thanks{The research of Y.-T. Oh was supported by the National Research Foundation of Korea (NRF) Grant funded by the Korean Government (NRF-2018R1D1A1B07051048).}
\address{Department of Mathematics, Sogang University, Seoul 121-742, Republic of Korea \& Korea Institute for Advanced Study, Seoul 02455, Republic of Korea}
\email{ytoh@sogang.ac.kr}

\author[Euiyong Park]{Euiyong Park}
\thanks{The research of E. Park was supported by the National Research Foundation of Korea (NRF) Grant funded by the Korean Government (NRF-2017R1A1A1A05001058).}
\address{Department of Mathematics, University of Seoul, Seoul 02504, Republic of Korea \& Korea Institute for Advanced Study, Seoul 02455, Republic of Korea}
\email{epark@uos.ac.kr}

\date{\today}
\subjclass[2010]{05E18, 05E05, 05E10}
\keywords{crystals,  semistandard tableaux,  cyclic sieving phenomenon, promotion}

\begin{abstract}
In this paper, we study a new cyclic sieving phenomenon on the set $\SST_n(\la)$ of semistandard Young tableaux with the cyclic action $\cc$ arising from its $U_q(\mathfrak{sl}_n)$-crystal structure. 
We prove that if $\la$ is a Young diagram with $\ell(\la) < n$ and $\gcd( n, |\la| )=1$, 
then the triple 
$
\left( \SST_n(\la),  \CC, q^{- \kappa(\la)} s_\la(1,q, \ldots, q^{n-1}) \right)
$ 
exhibits the cyclic sieving phenomenon, 
where $\CC$ is the cyclic group generated by $\cc$.  We further investigate a connection between $\cc$ and the promotion $\pr$ and show the bicyclic sieving phenomenon given by $\cc$ and $\pr^n$ for hook shape.  
\end{abstract}

\maketitle


\vskip 2em

\section*{Introduction}

The  {\it cyclic sieving phenomenon} was introduced in 2004 by Reiner-Stanton-White in \cite{RSW04}. 
Let $X$ be  a finite set, with an action of a cyclic group $C$ of order $n$, and $f(q) $ a polynomial in $q$ with nonnegative integer coefficients.
For $d \in \Z_{>0}$, let $\omega_d$ be a $d$th primitive root of the unity. 
We say that $(X, C, f(q))$ exhibits the  cyclic sieving phenomenon if, for all $c\in C$, we have 
$$
\# X^c = f(\omega_{o(c)}),
$$
where $o(c)$ is the order of $c$ and $X^c$ is the fixed point set under the action of $c$.
Note that this condition is equivalent to the following:
$$
f(q)\equiv \sum_{l=0}^{n-1}a_l q^l \pmod {q^n-1},
$$
where $a_l$ counts the number of $C$-orbits on $X$ for which the stablilizer-order divides $l$.
It has since then been extensively investigated for various combinatorial objects with an action of a finite cyclic group 
including words, multisets, permutations, non-crossing partitions, lattice paths, tableaux (see \cite{S11} for details).

In \cite{Sch72, Sch77}, Schuzenberger introduced  the {\it promotion} operator $\pr$ on (semi)standard Young tableaux, which takes one (semi)standard Young tableau to another via jeu de taquin slides.
 Afterwards, it has been studied widely and now has become one of the important objects in various research areas (see \cite{Stan09}).
It is known that it has a finite order, but in the best knowledge of the authors, 
its order is still mysterious except a few cases such as rectangular or staircase Young diagrams \cite{Hai92, PW11}.

Given a Young diagram $\la$, let  $\SST_n(\la)$ be the set of semistandard Young tableaux of shape $\la$ with entries in $\{1,2,\ldots, n\}$.
In \cite{Rho10},  Rhoades proved representation-theoretically that if $\la$ is of rectangular shape,   
the triple 
$$
\left(\SST_n(\la), \langle {\pr} \rangle, q^{-\kappa(\la)}s_\la(1,q, \ldots, q^{n-1}) \right) 
$$ 
exhibits  the cyclic sieving phenomenon, where 
$\kappa$ is the statistic on $\la=(\la_1, \la_2, \ldots)$ given by $\kappa(\la)=\sum_{i\ge 1}(i-1)\la_i$, 
and $s_\la(1,q, \ldots, q^{n-1})$ is the principal specialization of the \emph{Schur polynomial} $s_\la(x_1, x_2, \ldots, x_n)$.
This result, however, is no longer valid outside rectangular shape in general,  which says that, for a non-rectangular shape, 
another appropriate operator other than $\pr$ should be considered if we stick to the principal specialization on $\SST_n(\la)$. 
In \cite{FK14}, Rhoades' result is refined in the following manner. 
Let $\la=(a^b)$ and $\alpha=(\alpha_1, \ldots, \alpha_n)$ be a composition of $ab$, such that $\alpha$ is invariant under $l$th cyclic shift, then 
the triple   
 $$
\left( \SST_n(\la, \alpha),  \langle \pr^l \rangle, q^{\frac 12(a^2b-(\al_1^2+\al_2^2+ \cdots + \al_n^2))}K_{\la, \alpha}(q) \right)
$$ 
exhibits the cyclic sieving phenomenon, where $K_{\la, \alpha}(q)$ is a \emph{Kostka-Foulkes} polynomial associated with $ \la$ and $\alpha$.
Unfortunately, outside rectangular case, no results similar to this seem to be known yet.

In the present paper, we investigate the cyclic sieving phenomenon on $\SST_n(\la)$  with a cyclic action arising from its \emph{crystal} structure (see Section \ref{Sec: Crystals} for crystals).
For this purpose, we first notice that 
 $\pr = \sigma_1 \sigma_2 \cdots \sigma_{n-1}$, where 
$\sigma_i$ is the $i$th \emph{Bender–Knuth involution} acting on $\SST_n(\la)$. 
In general,  $\sigma_i$'s do not satisfy braid relations.
We then note that $\SST_n(\la)$ has a $U_q(\mathfrak{sl}_n)$-crystal structure, thus it is equipped with an action of the Weyl group.
Hence it would be very natural to consider the operator $\cc := \cs_1\cs_2\cdots \cs_{n-1}$ on $\SST_n(\la)$, where $\cs_i$ are simple reflections in the Weyl group. 
The operator $\cc$ shares several similarities with $\pr$, for instance, it is easy to check 
 that $ \wt(\cc (T)) = \wt(\pr (T))= s_1s_2\cdots s_{n-1} (\wt(T)) $.
One of the most favorable features of $\cc$, compared with $\pr$,  might be that 
its order is given by $n$ for arbitrary shape $\la,$  whereas the order of $\pr$  is very difficult to compute.

In the viewpoint of crystal theory, by using the operator $\cc$ instead of $\pr$, we observe a new cyclic sieving phenomenon on $\SST_n (\la)$ beyond rectangular shape.
More precisely, we prove that if $\la$ is a Young diagram with $\ell(\la) < n$ and $\gcd( n, |\la| )=1$, 
then the triple 
$$
\left( \SST_n(\la),  \CC, q^{- \kappa(\la)} s_\la(1,q, \ldots, q^{n-1}) \right)
$$ 
exhibits the cyclic sieving phenomenon, 
where $\CC$ is the cyclic group generated by $\cc$ (see Theorem \ref {Thm: main}). 
There are several examples for which our cyclic sieving phenomenon hold without the condition  $\gcd(n, |\la|) = 1$,
and Remark \ref{non relatively-prime case} shows an example for another cyclic sieving phenomenon with a specialization of $s_\la$ other than the  principal specialization.
It would be an interesting problem to give a characterization of Young diagrams $\la$ such that $\left(\SST_n(\la), c, f(q) \right)$  exhibits a cyclic sieving phenomenon, where $f(q)$ is a suitable specialization of  $s_\la$ (multiplied by a $q$-power). 
We also remark that the cyclic sieving phenomenon on the set of isolated vertices of a tensor product $B^{\otimes m}$ of a crystal $B$ with a different cyclic operator was studied in \cite{Wes16}.

Next, we turn to the connection between $\cc$ and the $\pr$. 
For an $n$-tuple $\alpha \in \Z_{\ge 0}^n$, let $\SST_n(\la, \alpha) := \{ T\in \SST_n(\la) \mid \cont(T) = \alpha  \}$.
We denote by
$\cont(\la)$ the set of all contents of $T$ where $T$ varies over $\SST_n(\la) $, and
 by $\cont^+(\la) $ the set of all $\al = (a_1, \ldots,  a_{n}) \in \cont(\la)$ such that
$ a_1 \ge a_2 \ge \cdots \ge a_{n}$.
Notice that $\SST_n(\la, \alpha)$ is invariant under $\pr^n$ for any $ \alpha \in \cont(\la)$.
For clarity, denote by $\pr^n|_{\alpha}$ the restriction of $\pr^n$ to $\SST_n(\la, \alpha)$.

We here deal with the case where $\la$ is of hook shape or two-column shape. 
In these special cases, we show that $\pr^n$ commutes with $\cs_i$'s, thus $\pr^n$ commutes with $\cc$.
We then show that the order of $\pr^n$ on $\SST_n(\la)$ equals $\lcm\{ \mathfrak{o}_{\la}(\alpha) \mid \alpha \in  \cont^+(\la) \}$,
where $\mathfrak{o}_{\la}(\alpha)$ denotes the order of  $\pr^n|_{\alpha}$ and 
$\lcm \{ k_1, k_2,  \ldots, k_t \}$ the least common multiple of $k_1, k_2, \ldots, k_t$.
We next consider the \emph{bicyclic sieving phenomenon} on $\SST_n(\la)$
in case where $\la$ is of hook shape with $(n, |\la|)=1$ (see \cite[Section 9]{S11} for the definition). 
Let $\la = (N-m, 1^m)$ with $\gcd(n, N)=1$, and consider the polynomial
$$
S_\la(q, t) = q^{-\kappa(\la)} \sum_{\mu \vdash N}  t^{ A_\mu}  K_{\la, \mu}(t^{ \frac{\po}{\po_\mu}}) \cdot m_\mu ( 1,q, q^2, \ldots, q^{n-1})
$$
given in Theorem \ref{Thm: bi-CSP in hook shape}.
Here $m_\mu(x_1, x_2, \ldots, x_n) $ is the \emph{monomial symmetric polynomial} assocoated to $\mu$, and 
$K_{\la, \mu}(t)$ is the \emph{Kostka-Foulkes polynomial} associated with $\la$ and $\mu$.
Note that the evaluation  $S_\la(q, t) $ at $t=1$ is equal to $ q^{- \kappa(\la)} s_\la(1,q, \ldots, q^{n-1})$. 
We show that the triple $  (\SST_n(\la), \CC\times \PP, S_\la(q, t) ) $ exhibits the bicyclic sieving phenomenon,
where  $\PP$  is the cyclic group generated by $\pr^n$
(see Theorem \ref{Thm: bi-CSP in hook shape}).

This paper is organized as follows:
In Section \ref{Sec: Crystals}, we review briefly the crystal theory. In Section \ref{Sec: SST}, we recall the combinatorics of Young tableaux. 
In Section \ref{Sec: CSP}, we study the action of $\cc$ on $\SST_n(\la)$ and prove the triple $\left( \SST_n(\la),  \CC, q^{- \kappa(\la)} s_\la(1,q, \ldots, q^{n-1}) \right)$ exhibits the cyclic sieving phenomenon. 
In Section \ref{Sec: c and pr}, we investigate a connection between $\cc$ and $\pr$ and show the bicyclic sieving phenomenon given by $\cc$ and $\pr^n$ for hook shape.

\vskip 2em

\section{Crystals}  \label{Sec: Crystals}

 Let $I$ be a finite index set. A square matrix $\cmA = (a_{ij})_{i,j\in I}$ is called a  \emph{generalized Cartan matrix} if it satisfies (i) $a_{ii}=2$ for $i \in I$  and $a_{ij} \in \Z_{\le 0}$ for $i \neq j$, (ii) $a_{ij} = 0$ if and only if $a_{ji} = 0$,
(iii) there exists a diagonal matrix $D={\rm diag}(\mathsf d_i \mid i \in I)$ such that $D\cmA$ is symmetric.
A {\it Cartan datum} $ (\cmA,\wlP,\Pi,\wlP^\vee,\Pi^\vee) $
consists of
\begin{enumerate}
\item a generalized Cartan matrix $\cmA$,
\item a free abelian group $\wlP$, called the {\em weight lattice},
\item $\Pi = \{ \alpha_i \mid i\in I \} \subset \wlP$,
called the set of {\em simple roots},
\item $\wlP^{\vee}=
\Hom_{\Z}( \wlP, \Z )$, called the \emph{coweight lattice},
\item $\Pi^{\vee} =\{ h_i \in \wlP^\vee \mid i\in I\}$, called the set of {\em simple coroots},
\end{enumerate}
which satisfy
\begin{enumerate}
\item $\lan h_i, \alpha_j \ran = a_{ij}$ for $i,j \in I$,
\item $\Pi$ is linearly independent over $\Q$,
\item for each $i\in I$, there exists $\varpi_i \in \wlP$, called the \emph{fundamental weight}, such that $\lan h_j,\varpi_i \ran =\delta_{j,i}$ for all $j \in I$.
\end{enumerate}

We set  $ \rlQ := \bigoplus_{i \in I} \Z \alpha_i$, called the \emph{root lattice}, and  $\rlQ^+ :=\sum_{i\in I} \Z_{\ge0}\alpha_i$. We fix a nondegenerate symmetric bilinear form $( \cdot \, , \cdot )$ on $\h^* :=\Q \otimes_\Z \wlP$ satisfying
\begin{equation*}
(\alpha_i,\alpha_j)=\mathsf d_i a_{ij} \quad (i,j \in I),\quad \text{and } \quad  \lan h_i,  \lambda\ran = \dfrac{2 (\alpha_i,\lambda)}{(\alpha_i,\alpha_i)} \quad (\lambda \in \mathfrak h^*, \ i \in I).
\end{equation*}
Let us denote by $\wlP^+: =\{ \lambda \in \wlP \mid \lan h_i, \lambda\ran \ge 0 \ \text{for all }  \ i \in I \}$ the set of \emph{dominant integral weights}, and
define $\Ht (\beta) :=\sum_{i \in I} k_i $  for $\beta=\sum_{i \in I} k_i \alpha_i \in \rlQ^+$.
Let $\weyl$ be the \emph{Weyl group} associated with $\cmA$, which is generated by
$$
s_i(\la) = \la - \langle h_i, \la \rangle \alpha_i\qquad \text{ for $i\in I$ and } \la\in \wlP.
$$

Let $U_q(\g)$ be the \emph{quantum group} associated with the Cartan datum $(\cmA, \wlP,\wlP^\vee \Pi, \Pi^{\vee})$, which
is generated by $f_i$, $e_i$ $(i\in I)$ and $q^h$ $(h\in \wlP)$ with certain defining relations (see \cite[Chater 3]{HK02} for details).
The notion of \emph{crystals} was introduced in \cite{Kas90, Kas91, Kas93}. We refer the reader to \cite{BS17, HK02} for details.
\begin{defn}
A \emph{crystal} associated with $(\cmA, \wlP,\wlP^\vee \Pi, \Pi^{\vee})$ is a set $B$ together with the maps $\wt : B \rightarrow \wlP$, $\te_i, \tf_i: B \rightarrow B \cup \{0\}$, and $\ep_i, \ph_i : B \rightarrow \Z \cup \{ - \infty \}$ ($i\in I$)
satisfying the following properties:
\begin{enumerate}
\item $\ph_i(b) = \ep_i(b) + \langle h_i, \wt(b) \rangle$ for all $i\in I$,
\item $\wt(\te_i b) = \wt(b)+\alpha_i$ if $\te_i b\in B$,
\item  $\wt(\tf_i b) = \wt(b)-\alpha_i$ if $\tf_i b\in B$,
\item $\ep_i( \te_i b ) =  \ep_i (b)-1$, $\ph_i( \te_i b ) =  \ph_i (b)+1$ if $\te_i b\in B$,
\item $\ep_i( \tf_i b ) =  \ep_i (b)+1$, $\ph_i( \tf_i b ) =  \ph_i (b)-1$ if $\tf_i b\in B$,
\item $\tf_i b = b'$ if and only if $ b = \te_i b'$ for $b,b'\in B$ and $i\in I$,
\item if $\ph_i(b) = -\infty$ for $b\in B$, then $\te_i b = \tf_i = 0$.
\end{enumerate}
\end{defn}
For a crystal $B$, we set $B_\xi := \{ b\in B \mid \wt(b) = \xi \}$ so that $B = \sqcup_{\xi \in \wlP} B_\xi$.
Let
$$
\wt(B) := \{ \xi \in \wlP \mid B_\xi \ne \emptyset  \}.
$$
For a dominant integral weight $\La \in \wlP^+$, we denote by $B(\La)$ the crystal of the irreducible highest weight $U_q(\g)$-module $V_q(\La)$ with highest weight $\La$.
For $i\in I$, we define the bijection $\cs_i$ on $B(\La)$ by
\begin{align} \label{Eq: def of si}
\cs_i ( b) =
\left\{
\begin{array}{ll}
 \tf_i^{ \langle h_i, \wt(b) \rangle } b  & \text{ if }  \langle h_i, \wt(b) \rangle \ge 0,\\
\te_i^{ - \langle h_i, \wt(b) \rangle }b & \text{ if } \langle h_i, \wt(b) \rangle < 0.
\end{array}
\right.
\end{align}
Then the Weyl group $\weyl$ acts on the crystal $B(\La)$ in which the simple reflection $s_i$ acts via $\cs_i$ for $i\in I$ (see \cite[Chapter 2.5]{BS17} for details).
Note that
\begin{align} \label{Eq: wt and si}
\wt( \cs_i (b) ) = s_i (\wt(b)) \qquad \text{ for $i\in I$ and $b\in B(\La)$.}
\end{align}

The \emph{character} $\ch B(\La)$ of $B(\La)$ is defined by
$$
\ch B(\La) := \sum_{\xi \in \wt(B(\La))}  | B(\La)_\xi | e^\xi,
$$
where $| B(\La)_\xi |$ is the number of elements of $B(\La)_\xi$, and
$e^\xi$ are formal basis elements of the group algebra $\Q[\wlP]$ with the multiplication given by $e^\xi e^{\xi'} = e^{\xi+\xi'}$.
The \emph{$q$-dimension} of $B(\La)$ is given by
$$
\qdim B(\La) = \sum_{\xi \in \wt(B(\La))}  | B(\La)_\xi | q^{ \ev( \La - \xi )},
$$
where $\ev:\rlQ \to \mathbb Z$ is the map defined as follows:
\begin{align}
\ev(\beta) := \sum_{i\in I} b_i \qquad  \text{ for }\beta = \sum_{i\in I} b_i \al_i \in \rlQ .
\end{align}

We now assume that $I = \{ 1,2, \ldots, r \}$ and  the Cartan matrix $\cmA$ is of \emph{finite type}. Note that the crystal $B(\La)$ is a finite set.
We define the bijection $\cc$ on $B(\La)$ as follows:
\begin{align} \label{Eq: def of c}
\cc := \cs_1\cs_2\cdots \cs_r.
\end{align}
Since $\cs_i$'s act on $B(\Lambda)$ as simple reflections of the Weyl group $\weyl$, $\cc$ can be viewed as a \emph{Coxeter element} of $\weyl$.
Let $\CC := \langle \cc \rangle$ be the cyclic subgroup of $\weyl$ generated by $\cc$, and $h$ the \emph{Coxeter number} of  $\weyl$.
\begin{lem} \label{Lem: C n crystal}
The cyclic group $\CC$ has order $h$ and acts on the crystal $B(\Lambda)$.
\end{lem}

\vskip 2em

\section{Semistandard tableaux} \label{Sec: SST}

For a partition $\la = (\la_1 \ge \la_2 \ge \ldots \ge \la_l > 0)$, 
the \textit{length $\ell(\la)$} of $\la$ is defined to be the number of positive parts of $\la$ and the \textit{size $|\la|$} of $\la$ the sum of all parts, that is, $\ell(\la) = l$ and $|\la| = \Sigma \la_i$.
Throughout this paper, we will confuse $\la$ with its Young diagram drawn in English convention, more precisely, an array of boxes in which the $i$th row has $\lambda_i$ boxes from top to bottom.  
The \emph{conjugate} $\la'$ of $\la$ denotes the Young diagram obtained from $\la$ by flipping the diagonal.

A {\it semistandard tableau} $T$ of shape $\la$ with entries bounded by $n$ is a filling of boxes of $\la$ with entries in $\{ 1, 2,\ldots, n\}$ such that
\begin{enumerate}
\item the entries in each row are weakly increasing from left to right, and 
\item  the entries in each column are strictly increasing from top to bottom.
\end{enumerate}
Let $\sh(T)$ denote the shape of a semistandard tableau $T$  
and $\SST_n(\la)$ the set of all semistandard tableaux of shape $\la$ with entries bounded by $n$.
We say that $\mathbf{b} = (p,q) \in T$ if $\mathbf{b}$ is a box of $T$ at the $p$th row and the $q$th column, and denote by $T(\mathbf{b})$ the entry of the box $\mathbf{b}$.
For example,  the following is a semistandard tableau of shape  $\la = (8,5,2)$ with entries bounded by $5$:
$$
\xy
(0,12)*{};(48,12)*{} **\dir{-};
(0,6)*{};(48,6)*{} **\dir{-};
(0,0)*{};(30,0)*{} **\dir{-};
(0,-6)*{};(12,-6)*{} **\dir{-};
(0,12)*{};(0,-6)*{} **\dir{-};
(6,12)*{};(6,-6)*{} **\dir{-};
(12,12)*{};(12,-6)*{} **\dir{-};
(18,12)*{};(18,0)*{} **\dir{-};
(24,12)*{};(24,0)*{} **\dir{-};
(30,12)*{};(30,0)*{} **\dir{-};
(36,12)*{};(36,6)*{} **\dir{-};
(42,12)*{};(42,6)*{} **\dir{-};
(48,12)*{};(48,6)*{} **\dir{-};
(3,9)*{1}; (9,9)*{1}; (15,9)*{2}; (21,9)*{2}; (27,9)*{2}; (33,9)*{4}; (39,9)*{5};(45,9)*{5};
(3,3)*{2}; (9,3)*{3}; (15,3)*{3}; (21,3)*{3}; (27,3)*{5};
(3,-3)*{3}; (9,-3)*{4};
\endxy
$$
\vskip 0.5em

For $T \in \SST_n(\la)$, the \emph{content} $\cont(T)$ of $T$ is defined to be the $n$-tuple $(c_1, \ldots, c_n)$, where $c_k$ is the number of occurrences of $k$ in $T$. Setting 
$x^T := x_1^{c_1} \cdots x_n^{c_n}$, we define
the \emph{Schur polynomial}  
$$
s_\la(x_1, \ldots, x_n) := \sum_{T\in \SST_n(\la)} x^T.
$$

Next, we describe the \emph{promotion operator} $\pr$ on $\SST_n(\la)$. Let $T\in \SST_n(\la)$. If $T$ does not contain entries equal to $n$, then $\pr(T)$ is defined to be the tableau obtained from $T$ by increasing all the entries by 1.
Otherwise, replace every entry equal to $n$ with a dot, then by using jeu-de-taquin, slide the dots to the northwest corner from left to right and top to bottom.
Finally, replace all dots by 1's and increase all other entries by 1 to obtain $\pr(T)$.

\begin{Ex} Let $n=4$ and $\la = (3,3,1)$. The following is an illustration of the promotion on a tableaux $T \in \SST_4(\la)$. 
$$
\xymatrix{
T = { \small  \begin{tabular}{|c|c|c|}
     \hline
       1 & 1 &2 \\
\hline
      3 & 3 &4 \\
\hline
      4  \\
\cline{1-1}
   \end{tabular}
}
\ar[r] & 
{ \small  \begin{tabular}{|c|c|c|}
     \hline
       1 & 1 &2 \\
\hline
      3 & 3 & $\bullet$ \\
\hline
      $\bullet$  \\
\cline{1-1}
   \end{tabular}
}
\ar[r] & 
{ \small  \begin{tabular}{|c|c|c|}
     \hline
        $\bullet$ & 1 &2 \\
\hline
      1 & 3 & $\bullet$ \\
\hline
     3  \\
\cline{1-1}
   \end{tabular}
}
\ar[r] & 
{ \small  \begin{tabular}{|c|c|c|}
     \hline
        $\bullet$ &  $\bullet$ &2 \\
\hline
      1 & 1 & 3 \\
\hline
      3  \\
\cline{1-1}
   \end{tabular}
}
\ar[r] & 
{ \small  \begin{tabular}{|c|c|c|}
     \hline
        1 &  1 &3 \\
\hline
      2 & 2 & 4 \\
\hline
      4  \\
\cline{1-1}
   \end{tabular}
}=\pr(T)
}
$$
\end{Ex}

From now on, we assume that the Cartan matrix $\cmA$ is of type $A_{n-1}$, i.e., $U_q(\g) = U_q(\mathfrak{sl}_n) $, with $I = \{ 1, 2, \ldots, n-1 \}$.
 For $k =1, \ldots, n$, we set $\epsilon_k := (0, \ldots, 1, \ldots, 0) \in \Q^{n}$ to be the unit vector with the 1 in the $k$th position.
For $i\in I$, we set
$$
\alpha_i := \epsilon_i - \epsilon_{i+1} \quad \text{ and } \quad  \varpi_i := \sum_{k=1}^i \epsilon_i.
$$
Then we identify the weight lattice $\wlP$ with the $n-1$-dimensional subspace of $\Q^n$ orthogonal to the vector $\epsilon_1 + \cdots + \epsilon_n$.
Note that the bilinear form $( \cdot \, , \cdot )$ corresponds to the usual inner product and $s_i(\epsilon_j) = \epsilon_{s_i(j)}$ for $i\in I$, where 
the subscript $s_i$ denotes the simple transposition $(i,i+1)$ in the symmetric group $\sg_n$.

Let $\la = (\la_1\ge \cdots \ge \la_\ell > 0)$ be a Young diagram with $\ell(\la) < n$. Letting $\la' = (\la_1', \la_2', \ldots, \la_t')$, we set
$
\wt(\la) := \sum_{k=1}^t \varpi_{\la_k'} \in \wlP^+.
$
It is well-known that $\SST_n(\la)$ admits a $U_q(\mathfrak{sl}_n)$-crystal structure and
$$
 \SST_n(\la) \simeq B(\wt(\la))
$$
 as a $U_q(\mathfrak{sl}_n)$-crystal. We refer the reader to  \cite[Chapter 3]{BS17} and \cite[Chapter 7]{HK02} for details. Note that
$\wt(T) = c_1 \epsilon_1 + \cdots + c_n \epsilon_n $ for $T \in \SST_n(\la)$,  where $\cont(T) = (c_1, \ldots, c_n)$.
We remark that the \emph{principal specialization} of $s_\la(x_1, \ldots, x_n)$ is equal to the $q$-dimension of $B(\wt(\la))$ up to a power of $q$, more precisely,
\begin{align} \label{Eq: schur and char}
s_\la(1,q, \ldots, q^{n-1}) = q^{\kappa(\la)} \qdim B(\wt(\la)), \quad \text{where $\kappa(\lambda) = \sum_{k=1}^\ell (k-1) \la_k$.}
\end{align}
Since $\SST_n(\la)$ is a $U_q(\mathfrak{sl}_n)$-crystal, the operator $\cc$ defined as in $\eqref{Eq: def of c}$ acts on $\SST_n(\la)$. The lemma below follows from Lemma \ref{Lem: C n crystal} immediately.
\begin{lem}
The cyclic group $\CC$ has order $n$ and acts on the $U_q(\mathfrak{sl}_n)$-crystal $\SST_n(\la)$.
\end{lem}

\vskip 2em

\section{Cyclic sieving phenomenon} \label{Sec: CSP}

As before, assume that the Cartan matrix $\cmA$ is of type $A_{n-1}$.
Let $c := s_1 s_2\cdots s_{n-1} \in \sg_n $. Note that $\sg_n$ acts on the weight lattice $\wlP$.
In addition, 
from the definition of $\pr$ and $\cc$ it follows that
\begin{align}
\wt( \cc (T))  = \wt( \pr (T) )  = c ( \wt(T) ) \qquad \text{  for $T \in \SST_n(\la)$.}
\end{align}

\begin{lem} \label{Lem: cc and weight}  \
\begin{enumerate}
\item For $\beta \in \rlQ$, we have $$\ev(c (\beta)) \equiv  \ev(\beta)  \pmod n.$$
\item Let $ \La \in \wlP^+$ and $N = \left( \La, \alpha_1 + 2 \alpha_2 + \cdots + (n-1) \alpha_{n-1}  \right) $. Then
$$  \ev( c (\La)) \equiv \ev(\La) - N  \pmod n.
$$
\end{enumerate}
\end{lem}
\begin{proof}
(1) As $\ev$ is linear, it suffices to consider the case where $\beta = \alpha_i $ for $i\in I$. By a direct computation, we can derive that 
$$
c (\alpha_i) =
\left\{
\begin{array}{ll}
\alpha_{i+1}  & \text{ if }  i\ne n-1,\\
-\alpha_1 - \alpha_2 - \cdots - \alpha_{n-1} & \text{ if } i=n-1.
\end{array}
\right.
$$
This tells us that $\ev( c (\alpha_i)) \equiv  \ev(\alpha_i)  \pmod n$.

(2) As above, due to the linearity of $\ev$, we may assume that $\La = \varpi_i$ for $i\in I$. Note that $i = \left( \varpi_i, \alpha_1 + 2 \alpha_2 + \cdots + (n-1) \alpha_{n-1}  \right)  $.
It follows from the identity $ c \varpi_i =  \varpi_i - \alpha_i - \alpha_{i-1} - \cdots - \alpha_1 $ that
$$
 \ev( c(\varpi_i) ) = \ev(\varpi_i) - i,
$$
which justifies the assertion.
\end{proof}

For positive integers $a,b \in \Z_{>0}  $, we denote by $\gcd(a,b)$ the greatest common divisor of $a$ and $b$.
A subset $\{ a_1, a_2, \ldots , a_n \} \subset \Z$ is called a \emph{complete residue system modulo} $n$ if it has no two elements that are congruent modulo $n$.

\begin{lem} \label{Lem: orbit}
Let $\La\in \wlP^+$ and $N = \left( \La, \alpha_1 + 2 \alpha_2 + \cdots + (n-1) \alpha_{n-1}  \right)$.
Suppose that $ \gcd(n,N)=1 $. Then, for any $\xi \in \wt(B(\La))$, the set
 $ \{ \ev( \La- \xi),  \ev(\La-  c(\xi)), \ldots, \ev( \La-  c^{n-1} (\xi))  \} $ is a complete residue system modulo $n$.
\end{lem}
\begin{proof}
Let $\xi\in \wt( B(\La) )$. Then we can write as $\xi = \La - \beta$ for some $\beta \in \rlQ^+$. Since
\begin{align*}
 c (\La) - \La \in \rlQ \quad \text{ and } \quad  c^k (\beta) \in \rlQ \quad \text{ for } k \in \Z_{\ge0},
\end{align*}
Lemma \ref{Lem: cc and weight} implies that
\begin{align*}
\ev( c^k ( c-\id)\beta) &= \ev( ( c-\id) ( c^k \beta)) \equiv 0 \pmod n,  \text{ and }\\
 \ev( c^k ( c - \id) \La) & \equiv \ev( ( c - \id) \La) \equiv -N \pmod n  .
\end{align*}
Combining these congruences, we derive that 
$$
\ev(  c^k ( c-\id ) \xi ) = \ev(  c^k ( c- \id) (\La-\beta) ) \equiv -N \pmod n,
$$
and thus, for $t = 1, \ldots, n-1$,
$$
\ev( c^t (\xi)) - \ev(\xi) \equiv \ev( ( c^t - \id) (\xi))  \equiv \sum_{k=0}^{t-1} \ev (  c^k ( c - \id) \xi ) \equiv - t \cdot N \pmod n.
$$
Now, our assertion follows from the assumption $\gcd( n, N)=1$.
\end{proof}

For two polynomials $f(q)$ and $g(q)$, we write 
$f(q) \equiv_n g(q)$ if $f(q) - g(q)$ is divisible by $q^n-1$.
We are now ready to state the main result on the cyclic sieving phenomenon for semistandard tableaux.

\begin{thm} \label{Thm: main}
Assume that $\la$ is a Young diagram with $\ell(\la) < n$ and $\gcd( n, |\la| )=1$. Then we have
\begin{enumerate}
\item every orbit of $\SST_n(\la)$ under the action of $\CC$ is free, and 
\item the triple $( \SST_n(\la), \CC, q^{- \kappa(\la)} s_\la(1,q, \ldots, q^{n-1}) )$ exhibits the cyclic sieving phenomenon.
\end{enumerate}
\end{thm}
\begin{proof}
(1)
Let $\La = \wt(\la)$ and denote by $\OO(\la)$ the set of all orbits of $\SST_n(\la)$ under the action of $\CC$.
Set $$
\oo (\OO) := \{ \ev( \La - \wt( S )) \mid S \in \OO  \}$$
for each orbit $\OO \in \OO(\la)$.
Also, for $T \in \SST_n(\la)$, we set
$\OO (T) := \{ \cc^k (T) \mid k \in \Z_{\ge 0} \} \in \OO(\la).$
Since the cyclic group $\CC$ has order $n$, we can deduce that 
\begin{enumerate}
\item[(i)]  $| \OO(T) | $ divides $ n$, and 
\item[(ii)] $| \oo(\OO(T)) |  \le | \OO(T) |$. 
\end{enumerate}
But, since $ n \le  | \oo( \OO(T)) |$ due to Lemma \ref{Lem: orbit}, we can deduce that 
$$
| \oo( \OO(T)) | = |\OO (T)| = n,
$$
as reqired.

(2)
Note that
\begin{align*}
|\la| = (\wt(\la), \alpha_1 + 2\alpha_2 + \cdots + (n-1)\alpha_{n-1}).
\end{align*}
For $\OO \in \OO(\la)$, we define
$$
\dim_q(\OO) := \sum_{e \in \oo(\OO)} q^e.
$$
As $\gcd(n, |\la|)=1$, Lemma \ref{Lem: orbit} implies that
$$
\dim_q(\OO) \equiv_n q^{n-1} + q^{n-2}+ \cdots + q + 1$$
for any orbit $\OO \in \OO(\la).$
Combining this with the identity $ |\OO(\la)| =  \displaystyle \frac{|\SST_n(\la)|}{n}$, which follows from (1),
we derive that 
\begin{align*}
\dim_q B(\wt(\la)) = \sum_{\OO \in \OO(\la)} \dim_q(\OO) \equiv_n \frac{ \SST_n(\la) }{ n } (  q^{n-1} + q^{n-2}+ \cdots + q + 1 ).
\end{align*}
Now the assertion follows from Equation $\eqref{Eq: schur and char}$.
\end{proof}

\begin{Rmk} \label{non relatively-prime case}
Theorem \ref{Thm: main} does not hold necessarily true without the condition $\gcd(n, |\la|) = 1$.
To see this,  
consider the case where $n=5$,  $\la = (2,1^{3})$ and $\La :=\wt(\la) = \varpi_1 + \varpi_4 $. Then $ \gcd(n, |\la|)=5 \ne 1 $.
Since the crystal $B(\La)$ is the crystal of the adjoint representation of $\mathfrak{sl}_5$, we have
\begin{align*}
\wt(B(\La)) = \{ 0,   \pm ( \epsilon_i - \epsilon_j) \mid 1 \le i< j \le 5  \}, \qquad
|B(\La)_\xi| =
\left\{
\begin{array}{ll}
4  & \text{ if }  \xi = 0,\\
 1 & \text{ if } \xi \ne 0.
\end{array}
\right.
\end{align*}
It follows from the identity
$c(\epsilon_i) =
\left\{
\begin{array}{ll}
\epsilon_{i+1}  & \text{ if }  i \ne  5,\\
 \epsilon_1 & \text{ if } i=5,
\end{array}
\right.
$
that every orbit is free or consists of a singleton. 
One can easily see that 
the number of free orbits equals $4$ and the number of fixed points equals $4$.
By a direct computation, we have
\begin{align*}
q^{-\kappa(\la)} s_\la(1,q, q^2, q^3, q^4) &= 1 + 2q + 3q^2 + 4q^3 + 4q^4 +  4q^5 +  3q^6 + 2q^7 + q^8 \\
&\not\equiv_5  4 + 4(1+q+q^2+q^3+q^4).
\end{align*}
This says that the triple $(B(\La), \CC, q^{-\kappa(\la)} s_\la(1,q, q^2, q^3, q^4))$ does not exhibit the cyclic sieving phenomenon.

However, setting $\sigma(\la) := \sum_{i\in I} \frac{i(i-1)}{2} \la_i$, we can observe that 
\begin{align*}
q^{-\sigma(\la)} s_\la(1,q, q^3, q^6, q^{10}) &= 1+q+q^3+q^4 + q^5 + q^6 + 2q^7 + q^8 + q^9 + 4q^{10} \\
& \ \  + q^{11} + q^{12} + 2q^{13} + q^{14} + q^{15} + q^{16} + q^{17} + q^{19} + q^{20} \\
& \equiv_5 4 + 4(1+q+q^2+q^3+q^4).
\end{align*}
Hence, quite interestingly, the triple $(B(\La), \CC, q^{-\sigma(\la)} s_\la(1,q, q^3, q^6, q^{10}))$ exhibits the cyclic sieving phenomenon.

\end{Rmk}

\vskip 2em

\section{Commuting action with $\cc$} \label{Sec: c and pr}

Recall that $\pr$ is the promotion on $\SST_n(\la)$. For $T\in \SST_n(\la)$, let
$$
\OO_\pr(T) := \{ \pr^k(T) \mid k \in \Z_{\ge 0}  \} \subset \SST_n(\la).
$$

\begin{prop} \label{Prop: order pr n}
Let $\la$ be a Young diagram with $\ell(\la) < n$. Suppose that $\gcd( n, |\la| )=1$. Then we have
\begin{enumerate}
\item for any $T \in \SST_n(\la)$, $| \OO_\pr(T) |$ is divisible by $n$, and 
\item the order of $\pr$ on $\SST_n(\la)$ is divisible by $n$.
\end{enumerate}
\end{prop}
\begin{proof}
(1) Let $T \in \SST_n(\la)$ and set
$$
 \mathcal{T}  := \{  k \in \Z_{\ge 0} \mid \ev( \La - \wt( T ) ) \equiv \ev( \La - \wt( \pr^k(T) ) ) \pmod{n}  \}.
$$
Since $\wt( \cc (T))  = \wt( \pr (T) ) = c (\wt(T))$ and $n$ is the order of $c$, by Lemma \ref{Lem: orbit}, we see that 
$$
 \mathcal{T} = \{  kn  \mid k \in \Z_{\ge 0}  \}.
$$
Since $| \OO_\pr(T) | \in \mathcal{T}$ by definition, we have the assertion.

(2) It follows from (1) directly.
\end{proof}

\begin{lem} \label{Lem: pr2}
Let $\la$ be a Young diagram with $\ell(\la) < n$.
Suppose that $\la$ is of hook shape or two-column shape.
Then $ \cs_1 \cdot \pr^2 = \pr^2 \cdot \cs_{n-1}  $.
\end{lem}

\begin{proof}
To begin with, let us fix necessary notations for the proof.

For $k \in \Z_{>0}$ and $l \in \Z_{\ge 0}$, let ${k}^l := (\underbrace{k, \ldots, k}_{l})$.
For $\mathbf{i} = (i_1, \ldots, i_l ) \in \Z^{l}$, let $\mathbf{i}^{+t} := (i_1+t, \ldots, i_l+t)$, and we simply draw
{ \tiny
 \begin{tabular}{|c c c|}
     \hline
 & $\mathbf{i}$ &   \\
     \hline
   \end{tabular}
   }
(resp.\
{ \tiny
 \begin{tabular}{|c|}
     \hline
   \\
$\mathbf{i}$  \\
  \\
     \hline
   \end{tabular}
   })
for the one-row (resp.\ one-column) tableau with entries $ (i_1, \ldots, i_l )$.

For $T\in \SST_n(\la)$, we write $k \in T$ if $k$ appears in $T$ as an entry.  For $1 \le k \le n$,
we set $T_{\le k}$ to be the tableau obtained from $T$ by removing all boxes with entries in $\{ k+1, \ldots, n \}$.
We also define $T_{< k}$, $T_{\ge k}$ and $T_{> k}$ in a similar manner.

\vskip 1em

\textbf{(Hook shape case)}

We assume that $\la$ is of hook shape, and choose any $T \in \SST_n(\la)$. We denote by $c_1(T)$ (resp.\ $r_1(T)$) the first column (resp.\ the first row) of $T$.
It is obvious that $\cs_1 \cdot \pr^2(T) = \pr^2 \cdot  \cs_{n-1} (T) $ when $\sh(T_{\le n-2}) =\emptyset $. Thus we assume that
$ \sh(T_{\le n-2}) \ne \emptyset $. Let
\begin{align*}
&\text{ $x$ := the number of occurrences of $n-1$ in $r_1(T)$},\\
&\text{ $y$ := the number of  occurrences of $n$ in $r_1(T)$}.
\end{align*}

\textbf{ (Case 1)}\ Suppose that $n-1, n \notin c_1(T)$.
Then we can write $T$ and $\cs_{n-1}(T)$ as follows:
$$
\tiny
T =
 \begin{tabular}{|c|c|c|c|c|}
     \hline
       \multicolumn{1}{|c}{ }  &       \multicolumn{1}{ c}{ $\mathbf{i}$ }  &       & $ {(n-1)}^x$ & $ {n}^y$ \\
     \hline
      \\
      $\mathbf{j}$ \\
        \\
     \cline{1-1}
   \end{tabular}\ ,
   \qquad
   \cs_{n-1} (T) =
 \begin{tabular}{|c|c|c|c|c|}
     \hline
       \multicolumn{1}{|c}{ }  &       \multicolumn{1}{ c}{ $\mathbf{i}$ }  &       & $ {(n-1)}^y$ & $ {n}^x$ \\
     \hline
      \\
      $\mathbf{j}$ \\
        \\
     \cline{1-1}
   \end{tabular}\ .
$$
By a direct computation, we can see that 
$$
\tiny
\pr^2(T) =
 \begin{tabular}{|c|c|c|c|c|}
     \hline
       \multicolumn{1}{|c}{ }  &       \multicolumn{1}{ c}{ $ {1}^x$ }  &       & $ \quad {2}^y \quad $ & $\quad \mathbf{i}^{+2} \quad $  \\
     \hline
      \\
      $\mathbf{j}^{+2}$ \\
        \\
     \cline{1-1}
   \end{tabular}\ ,
   \quad
\pr^2\cdot \cs_{n-1}(T) =
 \begin{tabular}{|c|c|c|c|c|}
     \hline
       \multicolumn{1}{|c}{ }  &       \multicolumn{1}{ c}{ $ {1}^y$ }  &       & $ \quad {2}^x \quad $ & $\quad \mathbf{i}^{+2} \quad $  \\
     \hline
      \\
      $\mathbf{j}^{+2}$ \\
        \\
     \cline{1-1}
   \end{tabular}\ ,
$$
which verifies the assertion since $\cs_1$ exchanges the number of $1$ and $2$.

\textbf{ (Case 2)}\ Suppose that  $n-1 \in c_1(T)$, but $n \notin c_1(T)$.
We first consider the case where $y=0$. Then $T$ and $\cs_1(T)$ can be written as follows:
$$
\tiny
T =
 \begin{tabular}{|c|c|c|c|c|}
     \hline
    &   \multicolumn{1}{c}{ }  &       \multicolumn{1}{ c}{ $\mathbf{i}$ }  &       & $ {(n-1)}^x$  \\
     \cline{2-5}
      $\mathbf{j}$ \\
        \\
     \cline{1-1}
        $n-1$     \\
     \cline{1-1}
   \end{tabular}\ ,
   \qquad
  \cs_{n-1}(T) =
 \begin{tabular}{|c|c|c|c|c|}
     \hline
    &   \multicolumn{1}{c}{ }  &       \multicolumn{1}{ c}{ $\mathbf{i}$ }  &       & $ {n}^x$  \\
     \cline{2-5}
      $\mathbf{j}$ \\
        \\
     \cline{1-1}
        $n$     \\
     \cline{1-1}
   \end{tabular}\ .
$$
Thus we have
$$
\tiny
\pr^2(T) =
 \begin{tabular}{|c|c|c|c|c|}
     \hline
       \multicolumn{1}{|c}{ }  &       \multicolumn{1}{ c}{ $ {1}^{x+1}$ }  &       & $\quad \mathbf{i}^{+2} \quad $  \\
     \hline
      \\
      $\mathbf{j}^{+2}$ \\
        \\
     \cline{1-1}
   \end{tabular}\ ,
   \quad
\pr^2\cdot \cs_{n-1}(T) =
 \begin{tabular}{|c|c|c|c|c|}
     \hline
       \multicolumn{1}{|c}{ }  &       \multicolumn{1}{ c}{ $ {2}^{x+1}$ }  &       & $\quad \mathbf{i}^{+2} \quad $  \\
     \hline
      \\
      $\mathbf{j}^{+2}$ \\
        \\
     \cline{1-1}
   \end{tabular}\ ,
$$
which justifies the assertion as before.

In case where of $y \ne 0$, we can see that 
$$ \tiny
T =
 \begin{tabular}{|c|c|c|c|c|}
     \hline
       \multicolumn{1}{|c}{ }  &       \multicolumn{1}{ c}{ $\mathbf{i}$ }  &       & $ {(n-1)}^x$ & $ {n}^y$ \\
     \hline
      \\
      $\mathbf{j}$ \\
        \\
     \cline{1-1}
       $n-1$  \\
     \cline{1-1}
   \end{tabular}\ ,
   \qquad
\cs_{n-1} (T) =
 \begin{tabular}{|c|c|c|c|c|}
     \hline
       \multicolumn{1}{|c}{ }  &       \multicolumn{1}{ c}{ $\mathbf{i}$ }  &       & $ {(n-1)}^{y-1}$ & $ {n}^{x+1}$ \\
     \hline
      \\
      $\mathbf{j}$ \\
        \\
     \cline{1-1}
       $n-1$  \\
     \cline{1-1}
   \end{tabular}\ ,
$$
and thus 
$$
\tiny
\pr^2(T) =
 \begin{tabular}{|c|c|c|c|c|c|}
     \hline
       \multicolumn{1}{|c}{ }  &       \multicolumn{1}{ c}{ $ {1}^{x+1}$ }  &   &  $2^y$   & $\quad \mathbf{i}^{+2} \quad $  \\
     \hline
      \\
      $\mathbf{j}^{+2}$ \\
        \\
     \cline{1-1}
   \end{tabular}\ ,
   \quad
\pr^2\cdot \cs_{n-1}(T) =
 \begin{tabular}{|c|c|c|c|c|c|}
     \hline
       \multicolumn{1}{|c}{ }  &       \multicolumn{1}{ c}{ $ {1}^{y}$ }  &   & $2^{x+1}$    & $\quad \mathbf{i}^{+2} \quad $  \\
     \hline
      \\
      $\mathbf{j}^{+2}$ \\
        \\
     \cline{1-1}
   \end{tabular}\, 
$$
as required. 

\textbf{ (Case 3)}\ Suppose that  $n-1 \notin c_1(T)$, but $n \in c_1(T)$. Then $T$ is given as follows:
$$
\tiny
T =
 \begin{tabular}{|c|c|c|c|c|c|}
     \hline
    &   \multicolumn{1}{c}{ }  &       \multicolumn{1}{ c}{ $\mathbf{i}$ }  &       & $ {(n-1)}^x$ & $ {n}^y$ \\
     \cline{2-6}
      $\mathbf{j}$ \\
        \\
     \cline{1-1}
        $n$     \\
     \cline{1-1}
   \end{tabular}\ .
$$
If $x=0$, then we have
$$
\tiny
\pr^2(T) =
 \begin{tabular}{|c|c|c|c|c|}
     \hline
       \multicolumn{1}{|c}{ }  &       \multicolumn{1}{ c}{ $ {2}^{y+1}$ }  &       & $\quad \mathbf{i}^{+2} \quad $  \\
     \hline
      \\
      $\mathbf{j}^{+2}$ \\
        \\
     \cline{1-1}
   \end{tabular}\ ,
   \quad
\pr^2\cdot \cs_{n-1}(T) =
 \begin{tabular}{|c|c|c|c|c|}
     \hline
       \multicolumn{1}{|c}{ }  &       \multicolumn{1}{ c}{ $ {1}^{y+1}$ }  &       & $\quad \mathbf{i}^{+2} \quad $  \\
     \hline
      \\
      $\mathbf{j}^{+2}$ \\
        \\
     \cline{1-1}
   \end{tabular}\ ,
$$
as required. 

If $x \ne 0$, then we obtain
$$ \tiny
\cs_{n-1} (T) =
 \begin{tabular}{|c|c|c|c|c|c|}
     \hline
    &   \multicolumn{1}{ c}{ }  &       \multicolumn{1}{ c}{ $\mathbf{i}$ }  &       & $ {(n-1)}^{y+1}$ & $ {n}^{x-1}$ \\
     \cline{2-6}
      $\mathbf{j}$ \\
        \\
     \cline{1-1}
        $n$     \\
     \cline{1-1}
   \end{tabular}\ ,
$$
and thus 
$$
\tiny
\pr^2(T) =
 \begin{tabular}{|c|c|c|c|c|c|}
     \hline
       \multicolumn{1}{|c}{ }  &       \multicolumn{1}{ c}{ $ {1}^{x}$ }  &   &  $2^{y+1}$   & $\quad \mathbf{i}^{+2} \quad $  \\
     \hline
      \\
      $\mathbf{j}^{+2}$ \\
        \\
     \cline{1-1}
   \end{tabular}\ ,
   \quad
\pr^2\cdot \cs_{n-1}(T) =
 \begin{tabular}{|c|c|c|c|c|c|}
     \hline
       \multicolumn{1}{|c}{ }  &       \multicolumn{1}{ c}{ $ {1}^{y+1}$ }  &   & $2^{x}$    & $\quad \mathbf{i}^{+2} \quad $  \\
     \hline
      \\
      $\mathbf{j}^{+2}$ \\
        \\
     \cline{1-1}
   \end{tabular}\ ,
$$
as required.

\textbf{ (Case 4)}\ Suppose that  $n-1, n\in c_1(T)$. Then $T$ and $\cs_{n-1}(T)$ can be written as follows:
$$
\tiny
T =
 \begin{tabular}{|c|c|c|c|c|c|}
     \hline
    &   \multicolumn{1}{c}{ }  &       \multicolumn{1}{ c}{ $\mathbf{i}$ }  &       & $ {(n-1)}^x$ & $ {n}^y$ \\
     \cline{2-6}
      $\mathbf{j}$ \\
        \\
     \cline{1-1}
     $n-1$        \\
     \cline{1-1}
        $n$     \\
     \cline{1-1}
   \end{tabular}\ ,
   \qquad
   \cs_{n-1} (T) =
 \begin{tabular}{|c|c|c|c|c|c|}
     \hline
    &   \multicolumn{1}{c}{ }  &       \multicolumn{1}{ c}{ $\mathbf{i}$ }  &       & $ {(n-1)}^y$ & $ {n}^x$ \\
     \cline{2-6}
      $\mathbf{j}$ \\
        \\
     \cline{1-1}
     $n-1$        \\
     \cline{1-1}
        $n$     \\
     \cline{1-1}
   \end{tabular}\ .
$$
A direct computation yields that
$$
\tiny
\pr^2(T) =
 \begin{tabular}{|c|c|c|c|c|}
     \hline
       \multicolumn{1}{|c}{ }  &       \multicolumn{1}{ c}{ $ {1}^{x+1}$ }  &       & $ \quad {2}^y \quad $ & $\quad \mathbf{i}^{+2} \quad $  \\
     \hline
      2 \\
       \cline{1-1}
       \\
      $\mathbf{j}^{+2}$ \\
        \\
     \cline{1-1}
   \end{tabular}\ ,
   \quad
\pr^2\cdot \cs_{n-1}(T) =
 \begin{tabular}{|c|c|c|c|c|}
     \hline
       \multicolumn{1}{|c}{ }  &       \multicolumn{1}{ c}{ $ {1}^{y+1}$ }  &       & $ \quad {2}^{x} \quad $ & $\quad \mathbf{i}^{+2} \quad $  \\
     \hline
      2 \\
       \cline{1-1}
       \\
      $\mathbf{j}^{+2}$ \\
        \\
     \cline{1-1}
   \end{tabular}\ ,
$$
as required.

\vskip 1em

\textbf{(Two-columns shape case)}

We assume that $\la$ is of two-column shape and $T \in \SST_n(\la)$. Let
\begin{align*}
&\text{ $p$ := the number of occurrences of $n-1$ in $T$,} \\
&\text{ $q$ := the number of occurrences of $n$ in $T$.}
\end{align*}
If $p=q$, then there is nothing to prove since $\cs_{n-1}(T) = T$ and $  \cs_1 \cdot \pr^2(T) =  \pr^2(T)$. 
From now on, suppose that $p \ne q$. Then we have the following cases:
$$
(p,q) \in \{  (2,0), (0,2),  (0,1), (1,0),  (2,1), (1,2) \}.
$$

\textbf{ (Case 1)}\ Suppose that $(p,q) = (2,0)$ or $(0,2)$. Then $T$ can be written as follows:
$$
\tiny
T =
 \begin{tabular}{|c|c|}
     \hline
       &  \\
       $\mathbf{i}$ & $\mathbf{j} $  \\
       &  \\
       \cline{2-2}
        & $a$ \\
       \cline{2-2}
        \\
     \cline{1-1}
        $a$ \\
     \cline{1-1}
   \end{tabular}\ ,
$$
where $a=n-1$ or $n$. Applying $\pr^2$ to $T$, we have
$$
\tiny
\pr^2(T) =
 \begin{tabular}{|c|c|}
     \hline
      $b$ &$b$  \\
     \hline
       &  \\
       $\mathbf{i}^{+2}$ & $\mathbf{j}^{+2} $  \\
       &  \\
       \cline{2-2}
        \\
     \cline{1-1}
   \end{tabular}\ ,
$$
where $b=1$ or $2$, respectively. This shows that $ \cs_1 \cdot \pr^2 (T) = \pr^2 \cdot \cs_{n-1} (T) $.

\textbf{ (Case 2)}\ Suppose that $(p,q) = (0,1)$ or $(1,0)$.
We first consider the case where $(p,q)=(0,1)$. Then we can write $T$ as follows:
$$
\tiny
T =
 \begin{tabular}{|c|c|}
     \hline
       &  \\
       $\mathbf{i}$ & $\mathbf{j} $  \\
       &  \\
       \cline{2-2}
        \\
     \cline{1-1}
        $n$ \\
     \cline{1-1}
   \end{tabular}
\quad
\normalsize
\text{ or}
\quad
\tiny
 \begin{tabular}{|c|c|}
     \hline
       &  \\
       $\mathbf{i}$ & $\mathbf{j} $  \\
       &  \\
       \cline{2-2}
        & $n$ \\
       \cline{2-2}
        \\
     \cline{1-1}
   \end{tabular}\ .
$$
In either case, the equality $T_{< n-1} = (\cs_{n-1}(T))_{< {n-1}}$ holds. Thus, it is easy to see that
$$
(\pr^2(T)_{>2}) = (\pr^2 \cdot \cs_n(T)_{>2}),
\qquad
(\pr^2(T))_{\le 2} =
\tiny
 \begin{tabular}{|c|}
     \hline
       2  \\
\hline
   \end{tabular},
\qquad
(\pr^2 \cdot \cs_{n-1}(T))_{\le 2} =
\tiny
 \begin{tabular}{|c|}
     \hline
       1  \\
\hline
   \end{tabular},
$$
which implies that $ \cs_1 \cdot \pr^2 (T) = \pr^2 \cdot \cs_{n-1} (T) $.

The remaining case where $(p,q)=(1,0)$ can be proved in the same manner.

\textbf{ (Case 3)}\ Suppose that $(p,q) = (2,1)$ or $(1,2)$.
We first consider the case where $(p,q)=(2,1)$. Then $T$ can be written as follows:
$$
\tiny
T =
 \begin{tabular}{|c|c|}
     \hline
       &  \\
       $\mathbf{i}$ & $\mathbf{j} $  \\
       &  \\
       \cline{2-2}
        & $n-1$ \\
       \cline{2-2}
        \\
     \cline{1-1}
        $n-1$ \\
     \cline{1-1}
        $n$ \\
     \cline{1-1}
   \end{tabular}
\quad
\normalsize
\text{ or}
\quad
\tiny
 \begin{tabular}{|c|c|}
     \hline
       &  \\
       $\mathbf{i}$ & $\mathbf{j} $  \\
       &  \\
       \cline{2-2}
        & $n-1$ \\
       \cline{2-2}
        & $n$ \\
       \cline{2-2}
        \\
     \cline{1-1}
        $n-1$ \\
     \cline{1-1}
   \end{tabular}\ .
$$
Then we have
$$
\tiny
\cs_{n-1}(T) =
 \begin{tabular}{|c|c|}
     \hline
       &  \\
       $\mathbf{i}$ & $\mathbf{j} $  \\
       &  \\
       \cline{2-2}
        & $n$ \\
       \cline{2-2}
        \\
     \cline{1-1}
        $n-1$ \\
     \cline{1-1}
        $n$ \\
     \cline{1-1}
   \end{tabular}
\quad
\normalsize
\text{ or}
\quad
\tiny
 \begin{tabular}{|c|c|}
     \hline
       &  \\
       $\mathbf{i}$ & $\mathbf{j} $  \\
       &  \\
       \cline{2-2}
        & $n-1$ \\
       \cline{2-2}
        & $n$ \\
       \cline{2-2}
        \\
     \cline{1-1}
        $n$ \\
     \cline{1-1}
   \end{tabular}\ .
$$
respectively. By the same argument as in \textbf{(Case 1)} and \textbf{(Case 2)}, we have
$$
(\pr^2(T)_{>2}) = (\pr^2 \cdot \cs_n(T)_{>2}),
\qquad
(\pr^2(T))_{\le 2} =
\tiny
 \begin{tabular}{|c|c|}
     \hline
       1 & 1  \\
\hline
      2  \\
\cline{1-1}
   \end{tabular},
\qquad
(\pr^2 \cdot \cs_{n-1}(T))_{\le 2} =
\tiny
 \begin{tabular}{|c|c|}
     \hline
       1 & 2  \\
\hline
      2  \\
\cline{1-1}
   \end{tabular}\ .
$$
Thus, we have that $ \cs_1 \cdot \pr^2 (T) = \pr^2 \cdot \cs_{n-1} (T) $.

The remaining case where $(p,q)=(1,2)$ can be proved in the same manner.
\end{proof}

\begin{Rmk}
It should be remarked that the identity $ \cs_1 \cdot \pr^2 = \pr^2 \cdot \cs_{n-1}  $ is not true in general. Let us consider the case where $n=4$, $\la = (3,2,1)$ and
$$
T =
\tiny
 \begin{tabular}{|c|c|c|}
     \hline
       1 & 1 & 4  \\
\hline
      2 & 3 \\
\cline{1-2}
      3 \\
\cline{1-1}
   \end{tabular}\  \in \SST_4(\la).
$$
Then 
$$
\tiny
\pr^2 \cdot \cs_3  (T) =
 \begin{tabular}{|c|c|c|}
     \hline
       1 & 2 & 3  \\
\hline
      2 & 4 \\
\cline{1-2}
      3 \\
\cline{1-1}
   \end{tabular}\ 
\ne \ 
 \cs_1 \cdot \pr^2   (T) =
 \begin{tabular}{|c|c|c|}
     \hline
       1 & 2 & 3  \\
\hline
      2 & 3 \\
\cline{1-2}
      4 \\
\cline{1-1}
   \end{tabular}\ .
$$
\end{Rmk}

\begin{Rmk}
It should also be remarked that $ \tf_1 \cdot \pr^2 \ne \pr^2 \cdot \tf_{n-1}  $ even in the case of a hook shape or a two-column shape (see \cite[Proof of Proposition 3.2]{BST10}).
For example, we consider the case where $n=3$ and
$T =
\tiny
 \begin{tabular}{|c|c|}
     \hline
       1 & 2  \\
\hline
      2 \\
\cline{1-1}
   \end{tabular}\ .
$
Then it is easy to see that
$$
\tiny
\pr^2\cdot \tf_2 (T) =
 \begin{tabular}{|c|c|}
     \hline
       1 & 3  \\
\hline
      2 \\
\cline{1-1}
   \end{tabular}\
\ne \
 \begin{tabular}{|c|c|}
     \hline
       1 & 2  \\
\hline
      3 \\
\cline{1-1}
   \end{tabular}
= \tf_1\cdot \pr^2( T)).
$$
\end{Rmk}

\begin{lem} \label{Lem: commuting with si}
Let $\la$ be a Young diagram with $\ell(\la) < n$. Suppose that $\la$ is of hook shape or two-column shape.
Then we have
$$
\cs_i \cdot \pr^n = \pr^n \cdot \cs_i \quad \text{ for } i\in I.
$$
In particular, we have
$
 \cc \cdot \pr^n = \pr^n \cdot \cc.
$
\end{lem}
\begin{proof}
It was shown in \cite[Proposition 3.2]{BST10} that $ \tf_{i+1} \cdot \pr =   \pr \cdot  \tf_{i}$ and $ \te_{i+1} \cdot \pr =   \pr \cdot  \te_{i}$ for $i=1, \ldots, n-2$.
For $T \in \SST_n(\la)$, we have
\begin{align*}
\langle h_i, \wt(T) \rangle = \langle  c^{-1}(h_{i+1}), \wt(T) \rangle = \langle  h_{i+1}, c (\wt(T)) \rangle = \langle  h_{i+1},  \wt(\pr (T)) \rangle.
\end{align*}
Then, it follows from the definition $\eqref{Eq: def of si}$ that
$$
 \cs_{i+1}\cdot \pr  (T) = \pr \cdot \cs_i (T) \qquad  \text{ for }i=1, \ldots, n-2.
$$
By Lemma \ref{Lem: pr2},  we have
\begin{align*}
\pr^n \cdot \cs_i &= \pr^{i-1} \cdot \pr^2 \cdot \pr^{n-i-1} \cdot \cs_i = \pr^{i-1} \cdot \pr^2 \cdot \cs_{n-1} \cdot \pr^{n-i-1}
= \pr^{i-1} \cdot \cs_{1}  \cdot \pr^2 \cdot \pr^{n-i-1} \\
&=  \cs_{i} \cdot \pr^{i-1} \cdot \pr^2 \cdot  \pr^{n-i-1} =  \cs_i\cdot \pr^n,
\end{align*}
which completes the proof.
\end{proof}

In the following, we assume that
\begin{equation} \label{Eq: assumption}
\begin{aligned}
 \gcd(n, |\la|)=1\quad \text{ and } \quad  \cs_i \cdot \pr^n = \pr^n \cdot \cs_i \text{ for } i\in I.
\end{aligned}
\end{equation}
Let $\PP$ be the cyclic group generated by $\pr^n$ acting on $\SST_n(\la)$.
Then the product group $\CC \times \PP$ acts on $\SST_n(\la)$. For $T \in \SST_n(\la)$, we set
\begin{align*}
& \OO_\CC(T) := \{ \cc^a (T) \mid a \in \Z_{\ge0}  \} , \qquad  \OO_\PP(T) := \{ \pr^{b n} (T) \mid b \in \Z_{\ge0}  \}, \\
& \OO_{\CC\times \PP} (T) := \{ \cc^a \pr^{b n} (T) \mid a,b \in \Z_{\ge0}  \}.
\end{align*}
For an $n$-tuple $\alpha \in \Z_{\ge 0}^n$, let $\SST_n(\la, \alpha) := \{ T\in \SST_n(\la) \mid \cont(T) = \alpha  \}$.
We denote by
$\cont(\la)$ the set of all contents of $T$ where $T$ varies over $\SST_n(\la) $, and
 by $\cont^+(\la) $ the set of all $\al = (a_1, \ldots,  a_{n}) \in \cont(\la)$ such that
$ a_1 \ge a_2 \ge \cdots \ge a_{n}$.
Notice that $\SST_n(\la, \alpha)$ is invariant under $\pr^n$ for any $ \alpha \in \cont(\la)$.
For clarity, denote by $\pr^n|_{\alpha}$ the restriction of $\pr^n$ to $\SST_n(\la, \alpha)$.

\begin{thm} \label{Thm: order of pr}
Suppose that $\eqref{Eq: assumption}$ holds. Then the following holds.
\begin{enumerate}
\item For $T\in \SST_n(\la)$, $|\OO_{\pr}(T)| = | \OO_{\CC\times \PP} (T) |$.
\item For an $n$-tuple $\alpha \in \Z_{\ge 0}^n$ with $ \SST_n(\la, \alpha) \ne \emptyset$, let $\mathfrak{o}_{\la}(\alpha)$ be the order of $\pr^n|_{\alpha}$.
Then the order of $\pr$ on $\SST_n(\la)$ equals $n \cdot \lcm\{ \mathfrak{o}_{\la}(\alpha) \mid \alpha \in  \cont^+(\la) \}$,
where $\lcm \{ k_1, \ldots, k_t \}$ denotes the least common multiple of $k_1 ,\ldots, k_t$.
\end{enumerate}
\end{thm}
\begin{proof}
(1)
 Since $|\OO_\CC(T)  | = n$ and  the order of $\pr^n$ is given by $\displaystyle \frac{|\OO_{\pr}(T)|}{n}$ by Theorem \ref{Thm: main} together with Proposition \ref{Prop: order pr n}, we deduce that 
$$
|\OO_{\pr}(T)| = n \cdot \frac{|\OO_{\pr}(T)|}{n} = |\OO_\CC(T)  | \cdot | \OO_\PP(T)| = | \OO_{\CC\times \PP} (T) |.
$$

(2) By the assumption $\eqref{Eq: assumption}$, we have  that $\mathfrak{o}_{\la}(\alpha) = \mathfrak{o}_{\la}(\cs_i \cdot \alpha)$ for $i\in I$.
Thus, by (1), we have that 
$$
\text{the order of $\pr$} = n \cdot \lcm\{  \mathfrak{o}_{\la}(\alpha) \mid \al \in \cont(\la)\} = n \cdot \lcm\{ \mathfrak{o}_{\la}(\alpha) \mid \alpha \in  \cont^+(\la) \}.
$$
\end{proof}

\begin{Ex}
We consider the case where $n=6$ and $\la = (2,2,2,1)$. Then $\eqref{Eq: assumption}$ holds by Lemma \ref{Lem: commuting with si} and 
$
\cont^+(\la) = \{ \alpha_1 := (2,2,2,1,0,0),\  \alpha_2 := (2,2,1,1,1,0), \alpha_3 := (2,1,1,1,1,1) \}
$.
As $|\SST_6(\la, \al_1)|=1$, it follows that $\mathfrak{o}_{\la}(\alpha_1) = 1$. In the case of  $\SST_6(\la, \al_2)$, we have
$$
\xymatrix{
{ \tiny
 \begin{tabular}{|c|c|}
     \hline
       1 & 1  \\
\hline
      2 & 2 \\
\hline
      3 & 4 \\
\hline
5 \\
\cline{1-1}
   \end{tabular}
}
\ar@/_1pc/[r]_{\pr^6} & \ar@/_1pc/[l]_{\pr^6}
{ \tiny
 \begin{tabular}{|c|c|}
     \hline
       1 & 1  \\
\hline
      2 & 2 \\
\hline
      3 & 5 \\
\hline
4 \\
\cline{1-1}
   \end{tabular}\ ,
}
}
$$
which tells us that the order $\mathfrak{o}_{\la}(\alpha_2) = 2$.
Finally,  we can see that $\SST_6(\la, \al_3)$ is decomposed into the following two orbits:
$$
\xymatrix{
{ \tiny
 \begin{tabular}{|c|c|}
     \hline
       1 & 1  \\
\hline
      2 & 3 \\
\hline
      4 & 5 \\
\hline
6 \\
\cline{1-1}
   \end{tabular}
}
\ar@/_1pc/[r]_{\pr^6} &
{ \tiny
 \begin{tabular}{|c|c|}
     \hline
       1 & 1  \\
\hline
      2 & 4 \\
\hline
      3 & 6 \\
\hline
5 \\
\cline{1-1}
   \end{tabular}
}\ ,
\ar@/_1pc/[l]_{\pr^6} &
{ \tiny
 \begin{tabular}{|c|c|}
     \hline
       1 & 1  \\
\hline
      2 & 3 \\
\hline
      4 & 6 \\
\hline
5 \\
\cline{1-1}
   \end{tabular}
}
\ar[r]_{\pr^6} &
{ \tiny
 \begin{tabular}{|c|c|}
     \hline
       1 & 1  \\
\hline
      2 & 5 \\
\hline
      3 & 6 \\
\hline
4 \\
\cline{1-1}
   \end{tabular}
}
\ar[r]_{\pr^6} & \ar@/_2.5pc/[ll]_{\pr^6}
{ \tiny
 \begin{tabular}{|c|c|}
     \hline
       1 & 1  \\
\hline
      2 & 4 \\
\hline
      3 & 5 \\
\hline
6 \\
\cline{1-1}
   \end{tabular}
}\ .
}
$$
Thus $\mathfrak{o}_{\la}(\al_3) = 6$, and by Theorem \ref{Thm: order of pr}, the order of $\pr$ is given by $6 \cdot \lcm\{ 1,2,6 \} = 36$.
\end{Ex}

We now focus on the hook shape  $\la = (N-m, 1^m)$. In this case, a closed formula for the order of $\pr$ was given in \cite{BMS14}. 
\begin{thm}[\protect{\cite[Theorem 3.9]{BMS14}}] \label{Thm: order of pr in hook shape}
For a hook shape $\la = (N-m, 1^m)$, 
the order of $\pr$ on $\SST_n(\la)$ is given as follows: 
$$
\begin{cases}
n \ \ &\text{ if } n= m+1, \\
n \cdot \lcm\{ m+1, m+2,  \ldots, \min\{ n, N\} -1  \} \ \ &\text{ if } n > m+1.
\end{cases}
$$
\end{thm}

Suppose that $\gcd(n, N )=1$.  Let $\alpha = (a_1, \ldots, a_n) \in \Z_{\ge 0}^n$ and let $ \nz(\alpha) $ denote the number of nonzero entries in $\alpha$.
It was proved in \cite{BMS14} that the order of $\pr^n|_{\alpha}$ is given as 
$$
\begin{cases}
1 \ \ &\text{ if } \nz(\alpha)= m+1, \\
\nz(\alpha)-1 &\text{ if } \nz(\alpha) > m+1, 
\end{cases}
$$
and the triple
 $(\SST_n(\la, \alpha), \pr^n|_\al,  X(q) )$ exhibits the cyclic sieving phenomenon, where
$ X(q) = \left[\begin{array}{c}
\nz(\alpha)-1 \\
m
\end{array}
\right]_q$ is the \emph{$q$-binomial coefficient}.
For $\la, \mu \vdash N$,
let $m_\la(x_1, x_2, \ldots, x_n) $ be the \emph{monomial symmetric polynomial} assocoated to $\la$ and
let $K_{\la, \mu}(q)$ be the \emph{Kostka-Foulkes polynomial} associated with $\la$ and $\mu$ (see \cite{M95} for the definitions).
The following lemma is needed for the bicyclic sieving phenomenon on $\SST_n(\la)$, which can be proved straightforwardly. 

\begin{lem} \label{Lem: pullback}
Let $\varphi: \widetilde{C} \rightarrow C$ be a surjective homomorphism between finite cyclic groups. Suppose that the triple $(X, C, f(q))$
exhibits the cyclic sieving phenomenon. We set $d := |\widetilde{C}|/|C|$. Then the triple $(X, \widetilde{C}, f(q^d))$ also exhibits the cyclic sieving phenomenon via the homomorphism $\varphi$.
\end{lem} 

We now have the following bicyclic sieving phenomenon.
\begin{thm}  \label{Thm: bi-CSP in hook shape}
Let $\la = (N-m, 1^m)$ with $\gcd(n, N)=1$, and let $\po$ and $\po_\al$ be the orders of $\pr^n$ and $\pr^n|_\al$ respectively.
We set
$$
S_\la(q, t) := q^{-\kappa(\la)} \sum_{\mu \vdash N}  t^{ A_\mu}  K_{\la, \mu}(t^{ \frac{\po}{\po_\mu}}) \cdot m_\mu ( 1,q, q^2, \ldots, q^{n-1}),
$$
where $A_\mu = \frac{\po}{\po_\mu} \left(-\kappa(\mu) + m \cdot \mu_1' - \frac{m(m+1)}{2} \right) $ and
$ \kappa(\mu)$ is defined in $\eqref{Eq: schur and char}$.
Then the triple $  ( \SST_n(\la), \CC\times \PP, S_\la(q, t) ) $ exhibits the bicyclic sieving phenomenon.
\end{thm}
\begin{proof}
Let $X$ be a finite set on which a finite group $G$ acts. For $g \in G$, let  $X^g := \{ x \in X  \mid x =g\cdot x   \}$ and let $o(g)$ be the order of $g$.
Note that the symmetric group $\sg_n$ acts on $\Z_{\ge0}^n$ by place permutation, i.e., $s_i \cdot (a_1, \ldots, a_n) = ( a_{s_i(1)}, \ldots, a_{s_i(n)} )$ for $i=1, \ldots, n-1$.

Let $c := s_1 \cdots s_{n-1} \in \sg_n $ and choose any $\mu = (\mu_1, \ldots, \mu_n) \vdash N$ with $\SST_n(\la,\mu) \ne 0$.
Let $ W(\mu) := \{  w \cdot \mu \mid w \in \sg_n   \} $ and set $l := \nz(\mu)$. Note that $\nz(\mu) = \mu_1'$.
It follows from Lemma \ref{Lem: orbit} and Theorem \ref{Thm: main} that
\begin{align} \label{Eq: distinct}
\text{ the elements $\mu$, $c \cdot \mu$, \ldots, $c^{n-1} \cdot \mu$ are all distinct.}
\end{align}

For $w = s_{i_1} \cdots s_{i_l} \in \sg_n$, we set $\cw := \cs_{i_1}\cdots \cs_{i_l}$ and define
$$
\OO_{\sg_n}(T) := \{ \cw \cdot T \mid w\in \sg_n  \} \quad \text{ for } T \in \SST_n(\la, \mu),
$$
and
$$
S(\mu):= \bigcup_{w \in \sg_n} \SST_n(\la, w \cdot \mu).
$$
Note that $\nz(\mu) = \nz(w \cdot \mu)$, $ \cw \cdot \SST_n(\la, \mu) = \SST_n(\la, w \cdot \mu)$ and $ |W(\mu)| = | \OO_{\sg_n}(T)| $ for $T \in \SST_n(\la, \mu)$.
Then the group $\CC\times \PP$ clearly acts on $S(\mu)$. It follows from $\eqref{Eq: distinct}$ together with Lemma \ref{Lem: commuting with si}
that
$$
S(\mu)^{ (\cc^a, \pr^{bn} )} = S(\mu)^{\cc^a} \cap S(\mu)^{\pr^{bn}} \quad \text{ for } 0 \le a < n, \ 0 \le b < \po_\mu.
$$
But, as every $\CC$-orbit of $\SST(\la)$ is free by Theorem \ref{Thm: main}, we can deduce that 
\begin{align} \label{Eq: S}
S(\mu)^{ (\cc^a, \pr^{bn} )} =
\left\{
\begin{array}{ll}
S(\mu)^{\pr^{bn}}  & \text{ if }  a=0,\\
\emptyset & \text{ if } a \ne 0.
\end{array}
\right.
\end{align}

Let
$$
X_\mu(t) = \left[\begin{array}{c}
l-1 \\
m
\end{array}
\right]_t
\quad \text{ and }\quad
Y_\mu(q) = q^{- \kappa(\la)} m_\mu(1,q, \ldots, q^{n-1}).
$$
Note that  the triple $( \SST_n(\la,  \mu), \pr^n|_\mu,  X_\mu(t) )$ exhibits the cyclic sieving phenomenon  by \cite[Theorem 4.3]{BMS14}.
By the same argument as in the proof of Theorem \ref{Thm: main}, for any $T\in \SST_n(\la, \mu)$, one can show that
$( \OO_{\sg_n}(T), \CC,  Y_\mu(q) )$ also 
exhibits the cyclic sieving phenomenon.
For $k \in \Z_{\ge 0}$, let $\omega_k$ be a primitive $k$th root of unity.
By the definition, we have
\begin{equation} \label{Eq: X Y}
\begin{aligned}
&X_\mu(\omega_{ o(\pr^{bn})}) = | \SST_n(\la,  \mu) ^{\pr^{bn}} | , \text{ and }\\
&Y_\mu( \omega_{ o(\cc^{a})} ) = |( \OO_{\sg_n}(T)^{\cc^a} | =  \delta_{a,0} \cdot | \OO_{\sg_n}(T) | = \delta_{a,0} \cdot | W(\mu) |.
\end{aligned}
\end{equation}
Here, the second equality for $ Y_\mu$ follows from the fact that every $\CC$-orbit is free.
Thus, by putting Lemma \ref{Lem: commuting with si},  $\eqref{Eq: S}$ and $\eqref{Eq: X Y}$ together, we can derive that 
\begin{align*}
|S(\mu)^{ (\cc^a, \pr^{bn} )} | = \delta_{a,0} \cdot |S(\mu)^{\pr^{bn}} | =  \delta_{a,0} \cdot |W(\mu)| \cdot | \SST_n(\la,  \mu) ^{\pr^{bn}} |
= X_\mu(\omega_{ o(\pr^{bn})}) \cdot Y_\mu( \omega_{ o(\cc^{a})} ),
\end{align*}
which tells us that the triple $ ( S(\mu), \CC\times \PP, X_\mu(t) \cdot Y_\mu( q ) ) $ exhibits the bicyclic sieving phenomenon.
Since $ \SST_n(\la) = \bigcup_{\mu \vdash N} S(\mu)$, we conclude that
$ ( \SST_n(\la), \CC\times \PP, \sum_{\mu \vdash N}  X_\mu(t^{{\po}/{\po_\mu}}) \cdot Y_\mu( q ) ) $ exhibits the bicyclic sieving phenomenon by Lemma \ref{Lem: pullback}.
Now the assertion
follows from the equality
$
X_\mu(t) =  t^{ -\kappa(\mu) + m \cdot \mu_1' - \frac{m(m+1)}{2}}  K_{\la, \mu}(t)
$
(\cite[Example 4.2]{KR86} or \cite[Lemma 7.12]{Kiri00}).
   \end{proof}

\vskip 2em


\bibliographystyle{amsplain}


\end{document}